\documentclass[12pt,usenames,dvipsnames]{article}
\usepackage[a4paper]{anysize}\marginsize{2cm}{2cm}{2cm}{2cm}
 \pdfpagewidth=\paperwidth \pdfpageheight=\paperheight
\usepackage{amsfonts,amssymb,amsthm,amsmath,eucal,tabu,hyperref}
\usepackage{pgf}
\usepackage{array}
\usepackage{pstricks}
\usepackage{pstricks-add}

\theoremstyle{plain}
\newtheorem{thm}{Theorem}[section]
\newtheorem{theorem}[thm]{Theorem}

\newtheorem{lemma}[thm]{Lemma}
\newtheorem{proposition}[thm]{Proposition}

\theoremstyle{definition}
\newtheorem{definition}[thm]{Definition}
\newtheorem{remark}[thm]{Remark}

\newtheorem{property}[thm]{Property}

\title{Like$\mathbb{N}$s a point of view on natural numbers, II}

\author{Edward Tutaj}
\date{\today}

\begin{document}
\maketitle

\begin{abstract}

 {\small In this paper we continue our research on the concept
  of {\it liken}. This notion has been defined  as a sequence
  of non-negative real numbers, tending to infinity and   closed with respect to  addition in $\mathbb{R}$.
  The most important examples of likens  are clearly the set of natural numbers $\mathbb{N}$ with addition and the set of positive
  natural numbers $\mathbb{N}^{*}$ with multiplication, represented by a sequence
  $(\ln({n+1}))_0^{\infty}$.
  The set of all likens can be parameterized   by the points
 of some infinite dimensional, complete metric space. In this {\it "space of likens"}  we consider elements up to isomorphism and define
  {\it "properties of likens"} as such, that are isomorphism invariant. The main result of
 this work is
 a theorem characterizing  the liken ${\mathbb N}^*$ of natural
 numbers with multiplication in the space of all likens.}

\end{abstract}

\section{Introduction}\label{intro}

We will begin by recalling the content of the paper \cite{Tut}, which is necessary to formulate and prove the main result of this paper, i.e. Theorem
\ref{main}. As it was mentioned in \cite{Tut}, the notion of a liken may be considered as some way of talking about of the so-called {\it Beurling numbers}
\cite{Beu}. The family of all likens - say $\mathcal{H}$ - (we will say also {\it the space of likens}) described in \cite{Tut}, constitutes a kind of a {\it
natural environment} where "live" the two fundamental mathematical structures: $(\mathbb{N},+)$ (the natural numbers with addition) and
$(\mathbb{N}^{*},\cdot)$, (the natural numbers with multiplication) which - as mathematical structures - are ordered semigroups. Let us pay attention here,
that by a liken we mean a sub-semigroup of the additive semigroup $\mathbb{R}^{+}$, so we replace $\mathbb{N}^{*}$ by  the sequence $(\ln({n+1}))_0^{\infty}$,
however without changing the notation,  and we call elements of this last sequence also "the natural numbers". Although $\mathcal{H}$ is a rather big space
(infinite dimensional complete metric space), the most of likens seem to be of little interest and if they were brought to life in \cite{Tut}, it was only to
look at the liken $\mathbb N^{*}$ from a slightly different point of view.

 The exact definition of a liken (in different versions)  will be recalled below,
 and at the beginning it is enough to know, that a liken $\mathbb L$ is a strictly
increasing sequence $\mathbb L = (x_n)_0^\infty$ of real numbers, which is a sub-semigroup of the semigroup $\mathbb R^{+}$. Hence in each liken $\mathbb L$ we
have two types of mathematical structures inherited from $\mathbb R^{+}$, i.e. the algebraic structure of the sub-semigroup with addition and the structure of
the ordered space with respect to the inequality in $\mathbb R$. This make  possible to  define the isomorphism of likens as a bijection which preserves both
structures - algebraic and ordinal.

  Different details concerning the relation of the isomorphism of likens will be discussed
  in the next section.  It appears - and this is in a sense a typical situation -
that all "interesting" likens (infinitely generated and with uniqueness) are algebraically isomorphic to each other, and at the same time, they are always
isomorphic as ordered spaces to each other.  On the other hand, as it was proved in \cite{Tut}, they are isomorphic as likens if and only if their sets of
generators are homothetic. As it was mentioned above, this situation is "typical".  To  understand better the meaning of the term "typical", let us consider
the example of the family of all infinite dimensional, separable Banach spaces. Each two such spaces are isomorphic as vector spaces since they have the
(vector) bases of the same cardinality, and each two such spaces are homeomorphic as topological spaces by the theorem of Kadec-Anderson, but two such spaces
are isomorphic as Banach spaces only when there exists a linear isomorphism which is also the topological homeomorphism.
 It seems that the basic advantage of using the abstract language of likens lies
in
 the fact that we can formulate different properties of likens in this
 language and consequently distinguish between them. Roughly speaking, {\it a property of
 liken}
 is each property, which is preserved by isomorphisms of likens.
 Perhaps the most important of such properties of likens is that they are generated
 by their irreducible elements (just like natural numbers by prime numbers
 in the semigroup $\mathbb N^{*}$), but this property is common for all likens. We will provide
 non-trivial examples of a few such properties later in the paper, but for now
 let us note, for (a trivial) example, that property: {\it the element  $x_2$ is undecomposable} is
 fulfilled in $\mathbb N^{*}$ but is not true in $\mathbb N$.

 The main result of this paper is  Theorem \ref{main} which gives a
 characterization of the liken $\mathbb N^{*}$ among all likens.
 For this, first we will formulate two, among others,  properties of likens called {\it
 convexity}, denoted by (C), and {\it Ockham's  razor property} denoted
 by (OR). Then the main theorem states: {\it if a liken $\mathbb L$ has the
 properties (C) and (OR) then it is isomorphic to $\mathbb
 N^{*}$.}

  The paper is organized as follows. In  Section \ref{sekcja1} we recall
  some definitions, notations and theorems proved in \cite{Tut},
   which will be used in this paper. In fact the contents of
  Section \ref{sekcja1} is to be found in \cite{Tut}, but because of some small
  differences in notations it will be better to collect in  Section \ref{sekcja1} all we will need about likens  in
  this paper. In  Section \ref{sekcja2} we  formulate a number of
  general properties of likens, in particular the mentioned
  properties (C) and (OR). In   Section \ref{sekcja3} we  present the
  proof of  Theorem \ref{main}. In the last section we  formulate a
  number of remarks.

\section{Definitions, notations and the main results about
likens}\label{sekcja1}

In this paper, as in \cite{Tut}, we will use the following notations:
\begin{equation}\label{oznaczenieR+}
{\mathbb{R}}^{+}=[0,\infty),
\end{equation}
\begin{equation}\label{oznaczenieQ+}
{\mathbb{Q}}^{+}=[0,\infty)\cap \mathbb{Q},
\end{equation}
\begin{equation}\label{duzaprzestrzen}
{\mathbb{R}}^{\mathbb{N}}= \left\{\overrightarrow{a}=(a_i)_1^\infty: a_i\in \mathbb{R}\right\}.
\end{equation}
\begin{equation}\label{duzystozek}
({{\mathbb{R}}^{+}})^{\mathbb{N}}= \left\{\overrightarrow{a}\in {\mathbb{R}}^{\mathbb{N}}: a_i\geq 0\right\}.
\end{equation}
\begin{equation}\label{malaprzestrzen}
{\mathbb{R}}^{\mathbb{N}}_{0}=\left\{\overrightarrow{a}\in {\mathbb{R}}^{\mathbb{N}}:\exists {j}:i>j\Rightarrow a_i=0\right\}.
\end{equation}
\begin{equation}\label{przestrzenwymierna}
{\mathbb{Q}}^{\mathbb{N}}= \left\{\overrightarrow{a}=(a_i)_1^\infty: a_i\in \mathbb{Q}\right\}.
\end{equation}
\begin{equation}\label{stozekwymierny}
({{\mathbb{Q}}^{+}})^{\mathbb{N}}= \left\{\overrightarrow{a}\in {\mathbb{Q}}^{\mathbb{N}}: a_i\geq 0\right\}.
\end{equation}
\begin{equation}\label{malawymierna}
{\mathbb{Q}}^{\mathbb{N}}_{0}=\left\{\overrightarrow{m}\in {\mathbb{Q}}^{\mathbb{N}}:\exists {j}:i>j\Rightarrow a_i=0\right\}.
\end{equation}

\begin{equation}\label{przestrzennaturalna}
{\mathbb{N}}^{\mathbb{N}}_{0}=\left\{\overrightarrow{a}\in {\mathbb{N}}^{\mathbb{N}}:\exists {j}:i>j\Rightarrow a_i=0\right\}.
\end{equation}

 Moreover, for
$\overrightarrow{a}\in {\mathbb{R}}^{\mathbb{N}}$ , and for $\overrightarrow{m}\in {\mathbb{N}}^{\mathbb{N}}_{0}$ we set:

\begin{equation}\label{iloczynskalarny}
\langle \overrightarrow{a}, \overrightarrow{m}\rangle = m_1
 a_1+m_2 a_2+...
\end{equation}

Let us note, that although $\overrightarrow{a}$ may tend to infinity, the righthand side sum is always finite, since the sequence $\overrightarrow{m}$ in fact
is finite.

The definition of the liken given in \cite{Tut} is following

\vspace{2mm}

\begin{definition}\label{definicjalikena}

A liken $\mathbb{L}$ is a sequence $(x_n)_{0}^{\infty}$ of real numbers such that:

a) For all $n\in \mathbb{N}$ we have $0=x_{0}\leq x_n<x_{n+1}$,

b) For all $n\in \mathbb{N}\ni m$ there is $k\in \mathbb{N}$ such that $x_n + x_m = x_k$.

\end{definition}

As it was observed in \cite{Tut}, a liken  $ {\mathbb L}$ is an increasing  sequence of nonnegative real numbers, which is closed with respect to the addition
and tends to infinity.

\vspace{2mm}

Now we recall the notion of the isomorphism of likens.

\begin{definition}\label{definicjahomomorfizmu}
 Let $(\mathbb{G},+)$ be a semigroup and let $\mathbb{L}$ be a
liken. We will say that a map $\varphi:\mathbb{G}\longrightarrow \mathbb{L}$ is

  a)  an algebraic homomorphism, when
$\varphi(x+y)=\varphi(x)+\varphi(y)$,

  b) an algebraic monomorphism, when it is an injective
homomorphism,

  c)   an algebraic isomorphism, when it is a surjective
monomorphism.

\end{definition}
\vspace{2mm}

In particular we know now, what  it means that two likens $\mathbb{L}$ and $\mathbb{K}$ are algebraically isomorphic. It is also clear, that each two likens
are isomorphic as ordered spaces, since they are similar to the ordered space $(\mathbb{N},\leq)$.
 Let us mention, that the map $\mathbb K \ni x_n \rightarrow y_n\in \mathbb L$
 is not (in general) a homomorphism of likens, and let us mention also, that if
$\varphi:\mathbb{K}\longrightarrow \mathbb{L}$ is an ordinal
 isomorphism, then it is unique.
 Finally we set

\vspace{3mm}
\begin{definition}\label{definicjaizomorfizmu}

 Two likens $\mathbb{L}$ and $\mathbb{K}$ are isomorphic if the
(unique) ordinal isomorphism is also an algebraic homomorphism.

\end{definition}

A very important consequence of the axioms of liken is the existence of {\it undecomposable elements} (called also {\it irreducible elements} or {\it prime
elements}).

\vspace{2mm}

\begin{definition}\label{definicjaelementunieroz}
Let $\mathbb{L}$ be a liken and let $u\in \mathbb{L}$. We will say, that $u$ is undecomposable if
$$u=v+w, v\in \mathbb{L}\ni w \Longrightarrow v=0\vee w=0.$$
\end{definition}
\vspace{2mm}

As it was observed in \cite{Tut}

\begin{proposition}\label{istnienieelemnierozk}
Each liken $\mathbb{L}=(x_n)_0^{\infty}$ has at least one undecomposable element.
\end{proposition}

Also (see \cite{Tut})

\begin{proposition}\label{istnienierozkladu}
 Let $\mathbb{L}$ be a liken, and let ${\mathcal{P}}_{\mathbb{L}}$ be the set
of all undecomposable elements of $\mathbb{L}$. Then each element of $x\in \mathbb{L}$ can be written in the form
\begin{equation}\label{rozklad}
x=m_1\cdot a_1 + m_2\cdot a_2+...+m_k\cdot a_k,
\end{equation}

where $m_1,m_2,...,m_k\in \mathbb{N}$,  $a_1,a_2,...,a_k\in {\mathcal{P}}_{\mathbb{L}}$, and $k\in \mathbb{N}$.

\end{proposition}

One may ask now about the uniqueness of the representation from Proposition \ref{istnienierozkladu}. In general, as it was discussed in \cite{Tut}, such
representations are not unique. So the above Definition  \ref{definicjalikena} of a liken admits likens without uniqueness. This is for example the case of the
so-called {\it numerical semigroups} \cite{Ros}
 with the associated  {\it Ap$\acute{e}$ry sets } (see also   Remark 1 in   Section \ref{sekcja4}). However in this paper we will be interested
only in likens with uniqueness, so further, in this paper, "liken" means "liken with uniqueness". This implies, as it will be discussed later, that all likens
are isomorphic algebraically. We recall below shortly the description of this situation presented widely in \cite{Tut}.

 Let $\mathcal{E}={\mathbb{N}}_{0}^{\mathbb{N}}$ denote, as in (\ref{przestrzennaturalna}), the set of
all sequences of natural numbers, with almost all terms vanishing, ie.:
\begin{equation}\label{bigoplus}
{\mathbb{N}}_{0}^{\mathbb{N}}:=\left\{\overrightarrow{n}=(n_1,n_2,...):(n_j\in {\mathbb N})\wedge(\exists i\in {\mathbb N}: k>i\Longrightarrow n_k=0\right\}).
\end{equation}

In the set $\mathcal{E}={\mathbb{N}}_{0}^{\mathbb{N}}$ we may consider the operations : $"+"$ - addition and $"\cdot"$ - multiplication by natural numbers -
defined as usually in a cartesian product. With these operations ${\mathbb{N}}_{0}^{\mathbb{N}}$ is an algebraic structure, which may be called {\it
semimodule} or {\it a cone} over ${\mathbb N}$.

We set : $e_k = (0,0,...,0,1,0,...,)$ , i.e. $e_k$ is an element of ${\mathbb{N}}_{0}^{\mathbb{N}}$ , with all terms  equal  $0$ except the $k-th$, which is 1.
So we have for ${\overrightarrow{n}}\in \mathcal{E}$:
\begin{equation}\label{bazabigoplusa}
{\overrightarrow{n}}=(n_1,n_2,...)= n_1\cdot e_1 + n_2\cdot e_2 + ...  .
\end{equation}

 Using the terminology from the linear algebra we may say, that
 $(e_k)_{1}^{\infty}$  {\it is a basis of the cone}
${\mathbb{N}}_{0}^{\mathbb{N}}$. This means precisely, that each element from ${\mathbb{N}}_{0}^{\mathbb{N}}$ can be, in a unique way, written as a linear
combination of $(e_k)_{k\in {\mathbb N}}$ with the coefficient from $\mathbb N$. Clearly $\mathcal{E}={\mathbb{N}}_{0}^{\mathbb{N}}$
 is a semigroup.

Clearly ${\mathbb R^{+}}$ is a cone over ${\mathbb N}$. A map
 $\varphi: \mathcal{E}\longrightarrow {\mathbb R^{+}}$
will be called  {\it a homomorphism of semigroups} (or of cones), when

\begin{equation}\label{homomorbigoplus}
\varphi(n_1\cdot e_1 + n_2\cdot e_2 +...)=n_1\cdot \varphi(e_1)+ n_2\cdot \varphi(e_2)+... .
\end{equation}

It is evident, that a homomorphism $\varphi:{\mathbb{N}}_{0}^{\mathbb{N}}\longrightarrow {\mathbb R^{+}}$ cannot be an ephimorphism,  since
${\mathbb{N}}_{0}^{\mathbb{N}}$ is countable,  but ${\mathbb R^{+}}$ is uncountable. However there exist monomorphisms
$\varphi:{\mathbb{N}}_{0}^{\mathbb{N}}\longrightarrow {\mathbb R^{+}}$, and the mentioned above {\it the space of likens} can be considered as the space of all
such monomorphisms.
 We will now give  description of this situation (for details see  \cite{Tut}).

\begin{proposition}\label{rozszerzenie}
Each function $a: {\mathbb N}\longrightarrow {\mathbb R^{+}}$ can be "extended" in a unique way to a homomorphism
$\tilde{a}:{\mathbb{N}}_{0}^{\mathbb{N}}\longrightarrow {\mathbb R^{+}}$ "by linearity" (i.e. $\tilde{a}({\overrightarrow{n}}) = \langle \overrightarrow{a},
\overrightarrow{n}\rangle$)
\end{proposition}

Now we will recall the description of the variety of infinitely generated likens. Let $\mathbb{L}$ be a liken and let the set ${\mathcal{P}}_{\mathbb{L}}=
\left\{a_1,a_2,...\right\}$ be infinite. We will assume, that $\mathcal{P}_{\mathbb{L}}$ is linearly independent in the vector space $(\mathbb{R},\mathbb{Q})$
of real numbers over rational numbers. This assumption is sufficient to have the uniqueness of the representation (\ref{rozklad}). We have observed in
\cite{Tut} that in the case when $\mathcal{P}_{\mathbb{L}}$ is infinite we must have $\lim_{k\rightarrow \infty}a_k=+\infty$. The converse is also true. Namely
\begin{proposition}\label{nieskonczeniewielegen}
Let $\overrightarrow{a}=(a_i)_1^{\infty}$ be a sequence from $(\mathbb{R}^{+})^{\mathbb{N}}$, which is linearly independent in the vector space
$(\mathbb{R},\mathbb{Q})$ and tends to infinity. Then $\tilde{a}({\mathbb{N}}^{\mathbb{N}}_{0})$ is a liken.
\end{proposition}

The essence of the concept of "the space of likens", denoted above by $\mathcal{H}$, lies in  latter Propositions. Namely, as we see, there exists one to one
correspondence between  infinite dimensional likens with uniqueness and the sequences of positive numbers tending to infinity and linearly independent in the
vector space $(\mathbb R, \mathbb Q)$. If we abandon the assumption of uniqueness, the space $\mathcal{H}$ looks better from the topological point of view.
Some further details are to be found in \cite{Tut}.

At the end of this section we recall one more theorem from \cite{Tut}.

Suppose, that we have two sequences $\overrightarrow{a}=(a_k)_1^{\infty}$ and $\overrightarrow{b}=(b_k)_1^{\infty}$, which generates two likens with uniqueness
denoted by ${\mathbb{L}}_a$ and ${\mathbb{L}}_b$ respectively. We have the following
\begin{theorem}\label{isomorfizmy}
In the notations as above the likens ${\mathbb{L}}_a$ and ${\mathbb{L}}_b$ are isomorphic, if and only if there exists a positive number $\lambda$ such that
$\overrightarrow{a}=\lambda\cdot \overrightarrow{b}$
\end{theorem}

\begin{remark}\label{uwaga0}

  As we have observed above (Proposition \ref{istnienierozkladu}), given a set of generators
(finite or infinite)  $\overrightarrow{a}=(a_k)_1^\infty$, - the liken ${\mathbb L}_{a}$ does not depend on the sequence $(a_k)_1^\infty$ but depends only on
the set of its elements. The only property we need from $(a_k)_1^\infty$ is to be locally finite. Clearly each finite set is locally finite, and for infinite
sequences $\overrightarrow{a}=(a_k)_1^{\infty}$ of generators, it is evident, that such a sequence is locally finite if and only if $\lim_{k\rightarrow
+\infty}a_k = +\infty$.In other words  for  a liken ${\mathbb{L}}_a$ we can always assume (and we do it in particular in this paper)  that its sequence of
generators is strictly increasing.
\end{remark}

\section{The different properties of likens}\label{sekcja2}

Let $\mathbb L = (x_n)_0^{\infty}$ be a liken. As it was mentioned above, by a {\it property of likens} we will mean - roughly speaking - all the conditions
concerning likens and formulated only using the language (and properties) of the addition and order in $\mathbb R$ and the addition and order in $\mathbb N$.
Clearly, the properties of likens  are preserved by isomorphisms of likens. We present below a few  examples of such properties.  The method of construction of
these properties is as follows. We take into account a particular liken (for example ${\mathbb N}^{*}$), we take into account a particular property of this
liken (for example the twin primes conjecture), we formulate this property in the language of likens, and this way we obtain a "property of liken".

\begin{property}\label{dimension}

   {\it We will say that the dimension of
 $\mathbb L$ equals $k\in \mathbb N$ if $\mathbb L$ has exactly $k$
 irreducible elements}. In other words $dim(\mathbb L)=k
 \Longleftrightarrow card ({\mathcal{P}}_{\mathbb L})=k$.  Such a liken
 will be said {\it finitely generated}.
\end{property}

Let us mention here, that $\mathbb N$ is a one dimensional liken.  The so-called {\it numerical semigroups} (for definition see \cite{Ros}) are finite
dimensional likens. A numerical semi-group is a semigroup generated in $({\mathbb N}, +)$ by the complements of finite sets. For example the set ${\mathbb
N}\setminus \left\{1,2\right\}$ is a numerical semigroup, which is - a three dimensional liken - (its generators are $\left\{3,4,5\right\}$). This liken is a
liken without uniqueness, since, for example $8=3+5=4+4$.

This way we have

\begin{property}\label{uniqueness}

 Suppose that $\overrightarrow{a}=(a_k)_1^{\infty}$ is a set of
 generators (finite or infinite) of the liken $\mathbb L$.  As we have mentioned above, {\it the liken
 $\mathbb L$ has the uniqueness property if for each $x\in \mathbb
 L$ there exists exactly one ${\overrightarrow{n}}\in
 {\mathbb{N}}_{0}^{\mathbb{N}}$ such that $x= \langle
 \overrightarrow{a},
 \overrightarrow{n}\rangle$.}

\end{property}

 \begin{property}\label{positionofgenerators}
  Suppose, that
 $\mathbb P =\left\{1<p_1<p_2< ....   \right\}$ is a subset of the set
 of natural numbers (finite, or infinite). {\it We will say that
 $\mathbb L$ has its generators exactly in ${\mathbb P}$ when for each
 $n\in \mathbb N$
 we have: $x_n\in \mathbb L$ is irreducible
 if and only if $n\in \mathbb P$.}

\end{property}

 It is not hard to see, that  if $\mathbb P$ is finite then one
 always can find a liken which has the generators exactly in $\mathbb P$.
 When the set $\mathbb P$ is infinite then the problem of
 the existence of a liken which has its generators "exactly" in $\mathbb P$ is
 more complicated. There is an obvious necessary condition for such a property, namely
 the set
 $\mathbb N \setminus \mathbb P$ must be infinite, but we do not know
 any reasonable characterization of those $\mathbb P$ for which
 there exists a liken which has the generators in $\mathbb P$.

 Let us recall here, that  Definition \ref{definicjalikena} does not assure
 the uniqueness, and that in this paper, for simplicity, we mean
 {\it liken} as {\it liken with uniqueness}. If $\mathbb L$ is a liken with uniqueness then each element of
this liken can be identified with a sequence of its coefficients in the representation $$x= \langle \overrightarrow{a}, \overrightarrow{n}\rangle = n_1a_1 +
n_2a_2 +... = n_1(x)a_1+n_2(x)a_2+... .$$ We set $supp (x)= \left\{ i\in\mathbb N: n_i(x)\neq 0\right\}$ and we  call this set {\it the support of $x$.}

\begin{property}\label{relativelyprime}

{\it We will say that a liken $\mathbb L$ has a "disjoint support property" if for each $n\in \mathbb N$ we have $supp(x_n)\cap supp(x_{n+1}) = \emptyset$.}

\end{property}

 Clearly, if a liken $\mathbb L$ has
the disjoint support property, then it is
 infinitely dimensional. For example the liken $\mathbb N^{*}$ has this
 property. If  $i\in supp(x)$ than we will say that $a_i$ divides
 $x$ (in symbol $a_i|x$).

\begin{property}\label{parity}

  {\it We will say that $\mathbb L$ has the
 "parity property" if for each $n\in \mathbb N$ we have $
 (x_1|x_n)
 \Rightarrow \neg(x_1|x_{n+1})$.}

 \end{property}

 In \cite{Tut} we studied the
 sequence of {\it gaps in likens}, i.e. the sequence of
 differences $\delta_{\mathbb L}(k)=\delta_k = \delta(x_k)=x_{k+1}-x_k.$
  By the definition of liken
 the sequence $\delta_k$ is strictly positive and as it was
 observed in \cite{Tut}, if $dim(\mathbb L)\geq 2$ then
 $\lim_{k \rightarrow \infty} \delta_{\mathbb L}(k) = 0.$. However
 in general, in particular in the case of finite dimensional likens,
  the sequence $\delta_k$ is not strictly decreasing. On the other
  hand there are the likens, such that $\delta_{\mathbb L}(k)$ is
  strictly decreasing.  Since the property {\it $\delta_k$ is strictly
  decreasing} is equivalent to: {\it for each $k\in \mathbb N$ we have
  $\delta_k>\delta_{k+1}$ }or equivalently, $2x_{k+1}>x_k +
  x_{k+2}$,
  then we formulate the property of convexity as follows.

\begin{property}\label{convexity}

  {\it A liken
  $\mathbb L$ is said to be  convex if and only if for each
  $k\in \mathbb N$ the following inequality holds
  $$2x_{k+1}>x_k + x_{k+2}.$$}

\end{property}

  It is not hard to observe, that the liken $\mathbb L$ is convex
  if and only if the points $(k,x_k)$ lies on the graph of a
  concave function $f:[0,\infty) \rightarrow \mathbb R$.
  For example $\mathbb N^{*}$ is a convex liken, since the
  function $x\rightarrow \ln(x)$ is concave. Given a sequence $(x_n)_0^{\infty}$, the sequence
  of its gaps plays a role of the derivative of the given
  sequence, thus the condition that the sequence of gaps is
  strictly decreasing corresponds to the claiming, that the second
  derivative is negative.

  It is clear, that the convexity of a liken is preserved by
  isomorphisms, but one may give example of two convex likens
  $\mathbb K$ and $\mathbb L$ which are not isomorphic. Indeed,
  let $\mathbb L= \mathbb N^{*}=(x_n)_0^\infty$ and let
  $N^{**}=(y_n)_0^\infty=\mathbb K =
  (\ln(2n+1))_0^\infty$ be a liken of all odd natural numbers (see Introduction \ref{intro}). This
  likens are both convex, but are not isomorphic. Indeed $x_3$ is
  composed and $y_3$ is irreducible. The same is true if one
  considers the liken ${\mathbb K}_p= (\ln(pn+1))_0^\infty$ for
  $p=1,2,...$.

  Let us look at the liken
 $\mathbb N^{*}$ and notice that except for one case, of two
 consecutive elements of this liken, at most one is irreducible.
 The situation is different in the liken $\mathbb N^{**}$, where
 every twin primes (in ${\mathbb N}^{*}$) are consecutive elements of the liken $\mathbb N^{**}$. We can
 therefore formulate the following property of likens:

\begin{property}\label{separationproperty}
  {\it We will say
 that in liken $\mathbb L$ almost all irreducible elements are
 separated when the number  of such pairs $(x_n,x_{n+1})$ in which
 both elements are irreducible, is finite.}

 \end{property}

 It not difficult  to show, that there exist  likens  without the
 separation property. On the other hand, the example of $\mathbb
 N^{**}$ shows that the problem of proving that a particular liken
 has the separation property, may be very difficult.

\vspace{3mm}

 The name of the next property refers to an old
philosophical principle.  The so-called {\it Ockham's razor principle}
 states that {\it entities should not be multiplied beyond necessity}.

Before we formulate this property for likens, let us establish some notations.
  Suppose that
 $\mathbb L = (x_m)_0^\infty$ is a liken. For $n\in \mathbb N$ we
 set ${\mathbb L}^{(n)}= {\mathbb L}(x_1,x_2,...,x_n)$ i.e. ${\mathbb
 L}^{(n)}$ is a liken generated  by all elements not greater than
 $x_n$, which is clearly a sub-liken of $\mathbb L$. Let us observe that
${\mathbb L}^{(n)}=\mathbb L(a_1,a_2,...,a_k)$ where $(a_1,a_2,....a_k)$ are all irreducible elements such that $a_i\leq x_n$.

  Let
  \begin{equation}\label{definicjaz}
  z_n = min \left\{ x : x\in {\mathbb L}^{(n)}, x>x_n \right\}.
\end{equation}
\begin{property}\label{razorprinciple}

 We will say that a liken $\mathbb L$ has the Ockham's razor
 property if

 $$supp(x_n)\cap supp(z_n) = \emptyset \Longrightarrow
 x_{n+1}=z_n.$$

 \end{property}

For the sake of explaining the name of {\it Ockham's razor} let us consider the following: Suppose we want to construct a liken $\mathbb L$ with the {\it
disjoint support property} and the construction runs recursively.  Suppose we construct $x_n$ and want to construct $x_{n+1}$. We do this: we determine the
smallest element of the liken generated by the already constructed among bigger than $x_n$ and denote it with $z_n$. If the support of $z_n$ is disjoint with
the support of $x_n$ then we take $z_n$ as $x_{n+1}$. This is just the considered property. And what happens, when $supp(x_n)\cap supp(z_n) \neq \emptyset$?
Because  the "necessity" (for us ) is the disjoint support property, then we must "multiply the entities"  and set $x_{n+1}= a_{k+1}$. Let us notice here, that
if $x_{n+1}= a_{k+1}$ then necessarily $x_{n+2} = z_n$ since undoubtedly $z_n\in {\mathbb L}^{(n+1)}$ and $z_{n+1}= z_n$. Indeed, if $a_{k+1}\notin
supp(z_{n+1})$ then $z_{n+1} \in \mathbb L^{(n)}$ and by definition $z_n= z_{n+1}$. On the other hand $a_{k+1} \notin supp(x_{n+2})$ since in such  case the
disjoint support property  would be violated.

\begin{property}\label{bertrand}

 We will say, that a liken $\mathbb L((a_k)_1^{\infty})$ has {\it
the Bertrand property} when for each $n\in \mathbb N$ there exists $k \in \mathbb N$ such that $x_n\leq a_k \leq x_n + a_1$.

\end{property}

\begin{property}\label{legendre}

 We will say, that a liken $\mathbb L((a_k)_1^{\infty})$ has {\it
the Legendre property} when

$$\lim_{n\rightarrow \infty} \frac{ card\left\{ k: a_k\leq x_n
\right\}}{n} = 0.$$

\end{property}

 All  properties, (\ref{uniqueness}-\ref{legendre}), are true in the
 liken $\mathbb N^{*}$, so they are consistent. On the other hand
 it is obvious that the conjunctions of some of the properties on
 the list above imply other or even all of the others.

 In this situation it is natural to ask if
 there are other likens besides $\mathbb N^{*}$ that have all of
  properties listed above, or which of these properties characterize
 the liken of natural numbers with multiplication.

Note that both properties (C) and (OR) are fulfilled in ${\mathbb N}^{*}$ while ${\mathbb N}^{**}$ has property (C) and no property (OR). Indeed, in this case
we have (in multiplicative model): $x_1=3$, $x_2=5$, $z_2=9$ and $x_3= 7$. Hence convexity do not imply the Ockham's razor property. On the other hand, the
property (C) implies the disjoint support property. Indeed suppose that $x_{k+1}= x_p+a_i$ and $x_k= x_q +a_i$. Hence  $\delta(x_k) =x_{k+1}-x_k= x_p-x_q \geq
x_{q+1}-x_q = \delta(x_q)$. But this is impossible, since in convex likens $q<k$ implies $\delta(x_q)>\delta(x_k)$.

\maketitle
\section{The main theorem}\label{sekcja3}

  In this section we are going to prove the main theorem, which gives
  a characterization of the liken $\mathbb N^{*}$ in the space of
  all likens. Suppose, that  $\overrightarrow{a}=(a_k)_1^\infty$
  is a sequence of positive real numbers generating a liken ${\mathbb L}(\overrightarrow{a})$
  denoted shortly by ${\mathbb L}_a$. (Let us recall, that in this paper, "liken" means "liken with uniqueness").
  In this notations we
  formulate the main result of this paper as follows:
 \vspace{3mm}

\begin{theorem}\label{main}

 If the liken ${\mathbb L}_a$ is convex
  and has the Ockham's razor property, then it is isomorphic to the
  liken ${\mathbb N}^{*}$.

  \end{theorem}

  First we will make a number of observations, that we will use in
  the proof.

\vspace{2mm}
\subsection{Multiplicative notation}

Our Definition \ref{definicjalikena} of a liken determines, that a liken $\mathbb L = (x)_0^{\infty}$ is an increasing sequence of non-negative real numbers
closed under addition in $\mathbb R$. Consider a new sequence defined by the formula
\begin{equation}\label{kapelusz}
\widehat{x}_{n}= \exp({x_{n-1}}),
\end{equation}
 (for $n=1,2,...$). This sequence
$\widehat{\mathbb L}=(\widehat{x}_{n})_1^{\infty}$ is a strictly increasing sequence of positive real numbers closed with respect to the multiplication in
$\mathbb R$ and obviously

\begin{equation}
\widehat{x_p+x_q}=\widehat{x}_{p+1}\cdot \widehat{x}_{q+1}
\end{equation}

We may say, that $\widehat{\mathbb L}$ is the same liken as $\mathbb L$, but we write $"\cdot"$ instead of $"+"$. The number $0$ is replaced by  $1$ and
indices go from $1$ to $+\infty$. Conversely, if we have a liken $\widehat{\mathbb L}= ({\widehat{x}_n})_1^{\infty}$ with the multiplicative notation, than the
sequence $(x_n)_0^\infty$ defined by the formula $x_n =\ln(\widehat{x}_{n+1})$ for $n=0,1,...$ is a liken with additive notation. In consequence, if in an
additive liken $\mathbb L$ we consider the gaps $\delta _k = x_{k+1}- x_k$ then in the multiplicative version we use the fraction
$$\widehat{\delta}_{k}= \frac{\widehat{x}_{k+1}}{\widehat{x}_{k}},$$  and conversely the quotients are
replaced by the differences. Let us agree, that if there is a "hat" above the symbols referring to the liken $\mathbb L$ then the formulas refer to the
multiplicative model of $\mathbb L$.

\vspace{2mm}
\subsection{The isomorphism  "exponent"}

 Let us take into account the set
 $\mathcal{E}={\mathbb{N}}_{0}^{\mathbb{N}}$,
 called in the sequel {\it the space of exponents} and let ${\mathbb
 L}_a = (x_n)_0^{\infty}$ be a liken. Hence, as we have observed above,
 the map
\begin{equation}\label{wykladnik}
\Omega_{\mathbb L} : {\mathbb{N}}_{0}^{\mathbb{N}} \ni
 \overrightarrow{m} \longrightarrow \langle \overrightarrow{a},
 \overrightarrow{m}\rangle \in {\mathbb
 L}_a
\end{equation}

is a bijection and it is an isomorphism of semigroups.


 The inverse map
\begin{equation}\label{wykladnikbis}
\Omega_{\mathbb L}^{-1} : {\mathbb
 L}_a \ni x_n \longrightarrow \Omega_{\mathbb L}^{-1}(x_n)\in
 {\mathbb{N}}_{0}^{\mathbb{N}}
\end{equation}
is also a bijection and is an
 isomorphism of semigroups.

When we have another liken $\mathbb K_b=(y_n)_0^{\infty}$ then we can consider an analogous isomorphisms
\begin{equation}\label{wykladnikter}
 \Omega_{\mathbb K} : {\mathbb{N}}_{0}^{\mathbb{N}} \ni
 \overrightarrow{m} \longrightarrow \langle \overrightarrow{b},
 \overrightarrow{m}\rangle \in {\mathbb
 K}_b
\end{equation}
as well as

\begin{equation}\label{wyklatnikfour}
 \Omega_{\mathbb K}^{-1} : {\mathbb
 K}_b \ni y_n \longrightarrow \Omega_{\mathbb K}^{-1}(y_n)\in
 {\mathbb{N}}_{0}^{\mathbb{N}}.
\end{equation}

The superposition

\begin{equation}
\Psi_{\mathbb K,\mathbb L}: {\mathbb K_b}\ni y_n \longrightarrow \Omega_L(\Omega_K^{-1}(y_n)) \in {\mathbb
 L}_a
\end{equation}
is an algebraic isomorphism of the likens ${\mathbb K}_b$ and ${\mathbb L}_a$, which allows us to say, that each two infinitely generated likens are
algebraically isomorphic. Let us notice, that in the case, when the sequences of generators are strictly increasing, then the described isomorphism
$\Psi_{\mathbb K,\mathbb L}$ is unique.


Now we take as ${\mathbb K}_b$ the particular liken ${\mathbb N}^{*} = (\ln(n+1))_1^{\infty}$, denoted as $(y_n)_0^{\infty}$ and we consider the analogous
isomorphisms $\Omega_{\mathbb N^*}$ and $\Omega_{\mathbb N^*}^{-1}$. We will write simply $\Omega$, the when lower index is implied by the context.

The composed isomorphism $\Psi_{\mathbb K,\mathbb L}$ in this special case  will be denoted simply by $\Psi$. We have

\begin{equation}
\Psi: {\mathbb N^{*}}\ni y_n \longrightarrow \Omega_L(\Omega_N^{-1}(y_n)) \in {\mathbb
 L}_a = \Omega(\Omega^{-1}(y_n))\in {\mathbb L}_a.
\end{equation}


\vspace{2mm}
\subsection{The beginning of the inductive proof}

 As we see, the map $\Psi$ is an algebraic isomorphism of
$\mathbb N^{*}$ and $\mathbb L_a$.  It remains to show, that $\Psi$ is also ordinal. This last assertion will be proved by induction. In fact we want to prove,
that for each $n\in \mathbb N$ we have $\Psi(y_n)= x_n$. Since clearly $\Psi(y_0)= x_0$ then the induction step is: if $\Psi(y_k)=x_k$ for $k\leq n$ then
$\Psi(y_{n+1})= x_{n+1}$. Or, in other words we must  prove the implication:

\vspace{2mm}

\begin{theorem}\label{krokindukcyjny}

If for each $0\leq i<j\leq n $ the inequality $x_i<x_j$ is equivalent to the inequality $y_i<y_j$ then $\Psi(y_{n+1})=x_{n+1}$.
\end{theorem}

\vspace{2mm}

First we shall verify, that for "small" $n$ the function $\Psi$ has the claimed property. Clearly,  for $n=0$ we have $x_0=0$ (i.e.$\Psi(y_0)=x_0$) , as in
each liken. Although, from the formal point of view, this is not necessary, we will check in details that $\Psi(y_k)=x_k$ for a few initial $k\in \mathbb N$ in
order to see how the properties (C) and (OR) "work".

 {\it Case $n=1$}. It must be $x_1 = a_1$, since $x_1$ must be indecomposable.
Indeed suppose that $x_1=u+v$, where $u\in {\mathbb L}\ni v$ , $u>0$ and $v>0$. Hence $0<u<x_1$, but this is impossible, since $x_1$ is next after  $x_0$. In
other words $x_1= a_1$. Hence $\Psi(y_1)=x_1$. Let us observe, that the equality $\Psi(y_1)=x_1$ does not require any additional assumption (i.e. it is true in
all likens).

{\it Case $n=2$}. It must be $x_2= a_2$. Indeed $z(x_1)=2a_1$ (the definition of $z(x)$ is in \ref{definicjaz}) and we see, that the supports of $z(x_1)$ and
$x_1= a_1$ are not disjoint. Hence, by (OR) $x_2=a_2$. Then, clearly $x_1=a_1<a_2=x_2< 2a_1$.

{\it Case $n=3$}. We have clearly $x_2=a_2<2a_1<a_1+a_2$. Hence $z(x_2) = 2a_1$ is disjoint with $x_2=a_2$. In consequence $x_3= 2a_1$. Let us remark, that
here we use (OR).

{\it Case $n=4$}. We see, that $x_3=2a_1$ is still in $\mathbb L^{(2)}$. It is also easy to check, that $z(x_3)=a_1+a_2$. Since the support of $z(x_3)$ is not
disjoint with the support of $x_3$, then $x_4=a_3$. Here we use once more the (OR) property.

{\it Case $n=5$}. Clearly $x_4\in \mathbb L^{(3)}$ and $z(x_4) = a_1 + a_2$ (this follows from $a_1+a_2<2a_3$). We see that $z(x_4)$  has the support disjoint
with the support of $x_4= a_3$. Hence (by (OR) $x_5= a_1+a_2 $.

{\it Case} $n=6$. Since $2a_2<2a_3$ then $z(x_5)\in L^{2}$. It is clear, that $3a_1<2a_1+a_2<a_1+2a_2$. Using the property (C) for $n=2$ we obtain
$3a_1=a_1+2a_1=x_1+x_3<2x_2=2a_2$. In consequence $z(x_5)=3a_1$ . Since $z(x_5)$ is not disjoint with $x_5$, then $x_6=a_4$.

{\it Case $n=7$.} Here, as before, and as we will do later, we may apply a general remark: if $x_n<z(x_n)$ and $x_n$ and $z(x_n)$ are not disjoint, then from
(OR) we have: $x_{n+1}=a_{k+1}$ and $x_{n+2}= z(x_n)$. This follows from the inequality $z(x_n)<2a_{k+1}$. Thus $x_7=3a_1$

{\it Case}  $n=8$. It follows from the considerations  for $n=6$ and $n=7$ that $z(x_7)=2a_2$, hence $x_8=2a_2$.

{\it Case}$ n=9$. We are now in $\mathbb L^{(4)}$ , and we calculate $z(x_8)$, which belongs "a priori" to $\mathbb L^{(4)}$. But, we have
$$a_1+a_3-x_8=a_1+a_3-2a_2 =
a_1+2a_3-2a_2-a_3=a_1+2x_4-2a_2-a_3>$$
$$a_1+x_3+x_5-2a_2-a_3= a_1+2a_1+a_1+a_2-2a_2-a_3=4a_1-(a_2+a_3)=
2x_3-(x_2+x_4)>0.$$ Since $a_1+a_3<2a_1+a_2 $ (because $x_4<x_5$) and clearly $a_1+a_3<a_1+a_4$, then $z(x_8)= a_1+a_3 $. Since $z(x_8)$ and $x_8=2a_2$ are
disjoint, then $x_9 = a_1+a_3$.

{\it Case} $n=10$. Since $x_4<x_5$ then $x_9=a_1+a_3<2a_1+a_2$. Clearly $2a_1+a_2<a_1+a_4$ and $2a_1+a_2<a_2+a_3$. Then $z(x_9)=2a_1+a_2$ and hence,  $x_{10} =
a_5$.

We see, that for $0\leq n \leq 10$ the map $\Psi$ satisfies the claimed properties on isomorphism of likens .

\vspace{2mm}
\subsection{The induction step}

As we are used to the multiplicative structure of the $\mathbb N^{*}$ semigroup,  we will write the proof of the main Theorem \ref{main} in the multiplicative
convention of both likens ${\mathbb L}_a$ and ${\mathbb N}^{*}$. Moreover, the role played by even numbers in ${\mathbb N}^{*}$ incline to the some
reformulation of the inductive step. Let us say also, that $\widehat{x}_k$ is even when $\widehat{x}_1|\widehat{x}_k$. Let us recall, that
\begin{equation}
\widehat{\Psi}:\mathbb N\ni n\longrightarrow  \widehat{\Psi}(n) \in {\mathbb L}_a
\end{equation}

is the (unique) algebraic isomorphism of the considered likens, i.e. for each $i,j \in \mathbb N$ we have

\begin{equation}
\widehat{\Psi}(i\cdot j)= \widehat{\Psi}(i) \cdot \widehat{\Psi}(j).
\end{equation}

Thus to prove, that $\widehat{\Psi}$ is an isomorphism of likens we must prove that $\widehat{\Psi}$ is an order isomorphism, which means, as usually for
likens, that for each $i\in \mathbb N$ we have: $\widehat{\Psi}(i)= \widehat{x_i}$. So to prove  Theorem {\ref{main}} it is sufficient to prove the following
theorem ("even" version of the induction step):

\begin{theorem}\label{inductionstep}
Suppose, that $n\in \mathbb N$  and that for each $1\leq i \leq 2n$ we have $\widehat{\Psi}(i) = \widehat{x_i}$. Then $\widehat{\Psi}(2n+1) =
{\widehat{x}}_{2n+1}$ and $\widehat{\Psi}(2n+2)={\widehat{x}}_{2n+2}.$
\end{theorem}

 We will start by formulating a number of observations.

i). Let consider the elements $\widehat{\Psi}(2j)$ for $1\leq j \leq 2n$. Since $\widehat{\Psi}$ is an algebraic isomorphism, for each $j\leq 2n$ we have
$\widehat{\Psi}(2j)= \widehat{\Psi}(2)\cdot \widehat{\Psi}(j)= \widehat{x}_2\cdot \widehat{x}_j$. If $j\leq n$ then we can write (by induction hypothesis)
$\widehat{x}_2\cdot \widehat{x}_j= \widehat{x}_{2j}$. In particular $\widehat{x}_2\cdot \widehat{x}_n = \widehat{x}_{2n} $, but we cannot write {\it a priori}
$\widehat{x}_2\cdot \widehat{x}_{n+1}= \widehat{x}_{2n+2}$ since this is just one of conditions to prove.  However, all these elements $\widehat{x}_2\cdot
\widehat{x}_j$ are even and are  obviously in the liken ${\mathbb L}^{(2n)}$.

ii). Since $\widehat{x}_n<\widehat{x}_{n+1}$ then $\widehat{x}_{2n}= \widehat{x}_2 \cdot \widehat{x}_n <\widehat{x}_2\cdot\widehat{x}_{n+1}$. But in ${\mathbb
L}_a$ we have the disjoint support property, so we must have

\begin{equation}
\widehat{x}_{2n}< \widehat{x}_{2n+1}<\widehat{x}_{2}\cdot \widehat{x}_{n+1}
\end{equation}

In other words this means, that between $\widehat{x}_{2n}$ and $\widehat{x}_2\cdot\widehat{x}_{n+1}$ there are some elements of the liken ${\mathbb L}_a$ but
we do not now how many of these elements are there, and what they are.

iii). Let us consider the set

\begin{equation}\label{box}
\mathcal{D}= (\widehat{x}_{2}\cdot \widehat{x}_{n} , \widehat{x}_{2}\cdot \widehat{x}_{n+1})\cap {\mathbb L}^{(2n)}
\end{equation}
 and
let us call it "a box".

First we will prove that

\begin{lemma}\label{psi(2n+1)}

 If $2n+1$ is composed, then $\widehat{\Psi}(2n+1) \in
(\widehat{x}_{2n}, \widehat{x}_2\cdot\widehat{x}_{n+1}).$
\end{lemma}

\begin{proof}

Let us assume then that $2n+1=p\cdot q$. Then clearly $p\geq 2$ and $q\geq 2$ and we have to prove the following inequalities:
\begin{equation}\label{psi(2n+1)1}
\widehat{x}_{2n}< \widehat{\Psi}(2n+1)
\end{equation}

and
\begin{equation}\label{psi(2n+1)2}
\widehat{\Psi}(2n+1)< \widehat{x}_2\cdot\widehat{x}_{n+1}.
\end{equation}

The first inequality follows directly from the inductive assumption. Indeed, the inductive assumption says, in particular, that
$$\widehat{\Psi}:[1,2,...,2n]\longrightarrow
[1,\widehat{x}_2,...,\widehat{x}_{2n}]$$ is a bijection.  But $2n+1\notin [1,2,...,2n]$ then $\widehat{\Psi}(2n+1)\notin
[1,\widehat{x}_2,...,\widehat{x}_{2n}]$ and in consequence $\widehat{x}_{2n}< \widehat{\Psi}(2n+1)$.

 The proof of the second inequality is more complicated. Clearly
 we may assume that $p\leq q$ and since $p\cdot q$ is odd then
 $p$ and $q$ are both odd, and  we have the inequality
 $$ 3\leq p \leq q < n.$$ Indeed, suppose $q\geq n$. Then we
 have $2n+1=p\cdot q \geq 3\cdot n$ which is possible only in
 $n=1$ but in our case $n\geq q\geq 3$.

Let us denote $$A= \frac{\widehat{\Psi}(2n+2)}{\widehat{\Psi}(2n+1)}.$$ Our aim is to show that $A>1$.  We have (recall that $\widehat{\Psi} $ is an algebraic
isomorphism on the whole $\mathbb N^{*}$ and recall that the quotients corresponds to differences in the additive models).
$$A=
\frac{\widehat{\Psi}(2n+2)}{\widehat{\Psi}(2n+1)}= \frac{\widehat{\Psi}(2(n+1))}{\widehat{\Psi}(p\cdot q)}= \frac{\widehat{x}_2\cdot
\widehat{x}_{n+1}}{\widehat{x}_p\cdot\widehat{x}_q}.$$ Let us notice here, that in this moment we cannot write ${\widehat{x}_p\cdot\widehat{x}_q}=
{\widehat{x}_{pq}}$ since $pq>2n$.  But we know, that $p$ is odd, and then $p+1$ is even and $p+1\leq n$. Thus $p+1= 2s \leq n$ and then we may write
$\widehat{x}_{p+1}=\widehat{x}_2\cdot \widehat{x}_s$. So we may also write

$$A=
\frac{\widehat{\Psi}(2n+2)}{\widehat{\Psi}(2n+1)}= \frac{\widehat{\Psi}(2(n+1))}{\widehat{\Psi}(p\cdot q)}= \frac{\widehat{x}_2\cdot
\widehat{x}_{n+1}}{\widehat{x}_p\cdot\widehat{x}_q}= \frac{x_{p+1}}{x_p} \cdot \frac{\widehat{x}_2\cdot \widehat{x}_{n+1}}{\widehat{x}_2 \cdot  \widehat{x}_s
\cdot \widehat{x}_q}=\frac{x_{p+1}}{x_p} \cdot \frac{\widehat{x}_{n+1}}{ \widehat{x}_s \cdot \widehat{x}_q} $$

Here is the time to replace $\widehat{x}_s \cdot \widehat{x}_q$ by $\widehat{x}_{sq}$ but for this we must evaluate $sq$ from above. We have $pq = 2n+1$ and
$sq <pq = 2n+1$, hence $sq$ is a natural number satisfying $sq\leq 2n$. This is sufficient for our purposes (for the use the induction hypothesis) although a
more detailed analysis allows us to prove that $sq\leq \frac{3n}{2}.$ So, by induction hypothesis, we may write $\widehat{x}_s \cdot \widehat{x}_q=
\widehat{x}_{sq}$ and in consequence we obtain

$$A=
\frac{\widehat{x}_{p+1}}{\widehat{x}_p} \cdot \frac{\widehat{x}_{n+1}}{ \widehat{x}_s \cdot \widehat{x}_q}= \frac{\widehat{x}_{p+1}}{\widehat{x}_p} \cdot
\frac{\widehat{x}_{n+1}}{ \widehat{x}_{sq}}. $$

iv). Here we will need a simple lemma resulting from the convexity property. Suppose, that $\mathbb L = (x_n)_0^{\infty}$  is a liken (in additive convention).
For fixed $j$ we can consider the sequence $\delta^{j}(n)= x_{n+j}-x_{n}$. It appears, that in convex likens (for each $j$) such a sequence is also strictly
decreasing. First we have

\begin{lemma}\label{maleniegapsow}
Let  $\mathbb L = (x_n)_0^{\infty}$ be a liken satisfying  the convexity property and let $p$ and $q$ be arbitrary positive integers such that  $1\leq p < q$.
Then $ x_{q-1}-x_{p-1}
> x_q - x_p$.
\end{lemma}

Let us recall the notation $\delta(k)= x_{k+1}-x_k$ and recall that in a convex liken we have $\delta(k+1)<\delta(k)$. Hence
$$ x_q-x_p = x_q-x_{q-1}+x_{q-1}-x_{q-2} +...+ x_{p+1}-x_p =$$
$$=\delta(q-1)+\delta)(q-2)+...+ \delta(p)
<\delta(q-2)+\delta(q-3)+...+\delta(p-1)=$$
$$=x_{q-1}-x_{q-2}+x_{q-2}-x_{q-3} +...+ x_p-x_{p-1} =
x_{q-1}-x_{p-1}.$$

From this lemma, by induction, we obtain the following inequality: if $1\leq p < q$ and $k\leq p$ then $x_{q-k} - x_{p-k} > x_q-x_p$.

The same, but in multiplicative notation, may be formulated as follows.

\begin{lemma}\label{maleniegapsow2}
 Let us suppose, that $x_p, x_q$ are two elements
  of a convex  liken $\mathbb L = (x_n)_1^{\infty}$ (in multiplicative convention), and  $1\leq k<p<q$. Then
   $$\frac{x_p}{x_q}>
  \frac{x_{p-k}}{x_{q-k}}.$$
\end{lemma}

v). Now we return to the evaluation from below of the quantity $A$. Our aim is to prove that $A>1$. We have proved that

$$A=
\frac{\widehat{x}_{p+1}}{\widehat{x}_p} \cdot \frac{\widehat{x}_{n+1}}{ \widehat{x}_s \cdot \widehat{x}_q}= \frac{\widehat{x}_{p+1}}{\widehat{x}_p} \cdot
\frac{\widehat{x}_{n+1}}{ \widehat{x}_{sq}}. $$

The inequality $A>1$ is evident when $n+1\leq sq$, then assume that $n+1 < sq$. By  Lemma \ref{maleniegapsow2} we have

$$A=\frac{\widehat{x}_{p+1}}{\widehat{x}_p} \cdot
\frac{\widehat{x}_{n+1}}{ \widehat{x}_{sq}}>\frac{\widehat{x}_{p+1}}{\widehat{x}_p} \cdot \frac{\widehat{x}_{n+1-s}}{ \widehat{x}_{sq-s}}. $$

As we have observed above, we have $n+1-s<2n$ and $sq-s<2n$ and additionally we have

$$\frac{n+1-s}{sq-s}= \frac{p}{p+1}.$$

Indeed we have the sequence of equivalent equalities

$$\frac{n+1-s}{sq-s}= \frac{p}{p+1} \Leftrightarrow
(p+1)(n+1-s)= p(sq-s)\Leftrightarrow 2s(n+1-s)=ps(q-1)\Leftrightarrow 2(n+1-s)=$$
$$=p(q-1)\Leftrightarrow 2n+2 -2s = pq -p \Leftrightarrow 2n+2 - p -
1 = 2n+1 - p.$$ The last equality is true since we assumed that $pq=2n+1$ and $p+1 = 2s.$

Hence there exists a number $t\in \mathbb N$ such that $n+1-s = tp$ and $sq-s = t(p+1)$.  From the induction assumption we have

$$A=\frac{\widehat{x}_{p+1}}{\widehat{x}_p} \cdot
\frac{\widehat{x}_{n+1}}{ \widehat{x}_{sq}}>\frac{\widehat{x}_{p+1}}{\widehat{x}_p} \cdot \frac{\widehat{x}_{n+1-s}}{ \widehat{x}_{sq-s}}=
\frac{\widehat{x}_{p+1}}{\widehat{x}_p} \cdot \frac{\widehat{x}_{tp}}{ \widehat{x}_{t(p+1)}} = \frac{\widehat{x}_{p+1}}{\widehat{x}_p} \cdot
\frac{\widehat{x}_{t}\cdot \widehat{x}_p}{ \widehat{x}_{t}\cdot \widehat{x}_{p+1}} = 1.$$

This ends the proof of the inequality (\ref{psi(2n+1)2}) and at the same time of  Lemma \ref{psi(2n+1)}.

\end{proof}

vi). Consider now the situation, when between $\widehat{x}_{2n}$ and  $\widehat{x}_2\cdot \widehat{x}_{n+1}$  there are no elements of the liken ${\mathbb
L}^{(2n)}$, i.e. the box is empty. In this case $z(2n)= \widehat{x}_2\cdot \widehat{x}_{n+1}$. The razor property implies then, that $\widehat{x}_{2n+1} =
a_{k+1}$. But in this case $2n+1$ cannot be composed, since, when $2n+1$ is composed, then ${\widehat{\Psi}(2n+1)}$ is in ${\mathbb L}^{(2n)}$, and, as we have
proved above
$${\widehat{\Psi}(2n+1)}\in
(\widehat{x}_{2n}, \widehat{x}_2\cdot \widehat{x}_{n+1}),$$ contrary to our assumption. Hence $2n+1 = p_{k+1}$ ($p_{k+1}$ is the (k+1)-th prime in $\mathbb
N^{*} $) and $ a_{k+1} = \widehat{x}_{2n+1},$ and we  see that in this case $\Psi(2n+1)=x_{2n+1}.$

vii). Summarizing, we have proved, that the element
 ${\widehat{\Psi}(2n+1)}$ is always in the interval
$(\widehat{x}_{2n}, \widehat{x}_2\cdot \widehat{x}_{n+1})$. To end the proof of the main Theorem it is enough to show, that in the interval $(\widehat{x}_{2n},
\widehat{x}_2\cdot \widehat{x}_{n+1})$ there are no other elements of ${\mathbb L}^{(2n)}$ besides, possibly, ${\widehat{\Psi}(2n+1)}$.

viii). Suppose that there exists an element $\widehat{x}$ such that $\widehat{x} \in {\widehat{\mathbb L}}^{(2n)} \cap (\widehat{x}_{2n}, \widehat{x}_2\cdot
\widehat{x}_{n+1})$ and $\widehat{x} \neq {\widehat{\Psi}(2n+1)}$. Since $\widehat{x} \in  {\widehat{\mathbb L}}^{(2n)}$ then there exist two natural numbers
$r$ and $s$, such that $ r\leq 2n$, $s\leq 2n$ , $\widehat{x} =\widehat{x}_r \cdot \widehat{x}_s$ and $\widehat{x}_{2n}< \widehat{x}_r \cdot \widehat{x}_s <
\widehat{x}_2\cdot \widehat{x}_{n+1}$. Since $\widehat{x}\neq {\widehat{\Psi}(2n+1)}$ than $r\cdot s > 2n+1$ and since $\widehat{x}_{2n}< \widehat{x}_r \cdot
\widehat{x}_s < \widehat{x}_2\cdot \widehat{x}_{n+1}$ then both  $r$ and $s$ are odd. Clearly, we can assume that $r\leq s$ and  observe, that in fact we have
the inequalities
\begin{equation}\label{jeden}
 3\leq r \leq s <n
\end{equation}

 Since $r>1$ and $r$ is odd, we have $r\geq 3$,
so we must show that $s<n$.  Suppose that $s\geq n$. But $s\leq 2n$ and $\widehat{\Psi}$ is increasing in the interval $[1,2n]$ (induction),  thus $
\widehat{x}_s \geq \widehat{x}_n$ and in consequence

\begin{equation}\label{dwa}
\widehat{x}_2\cdot \widehat{x}_{n+1}> \widehat{x}_r \cdot \widehat{x}_s> \widehat{x}_3 \cdot \widehat{x}_n
\end{equation}

This gives the inequality (in the multiplicative convention)
$$ \widehat{x}_2\cdot \widehat{x}_{n+1} > \widehat{x}_3 \cdot
\widehat{x}_n$$ and this gives  (in additive convention) the inequality $$ x_{n+1}-x_n > x_2-x_1$$ which impossible in convex likens. Since $ n> r\geq 3$ and
$r$ is odd, then $r-1$ is even and we can put $r-1= 2t$. In consequence we have

\begin{equation}\label{trzy}
\widehat{x}_2\cdot \widehat{x}_t \cdot \widehat{x}_s = \widehat{x}_{r-1}\cdot \widehat{x}_s <\widehat{x}_r \cdot \widehat{x}_s< \widehat{x}_2\cdot
\widehat{x}_{n+1}
\end{equation}

which gives the inequality

\begin{equation}\label{cztery}
\widehat{x}_t\cdot \widehat{x}_s < \widehat{x}_{n+1}.
\end{equation}

Since we are in the interval $[1, \widehat{x}_2,\widehat{x}_3,...., \widehat{x}_{2n}]$, and ${\widehat{\Psi}}^{-1}$ is increasing then $t\cdot s < n+1$.

ix). The end of our reasoning is similar as above. We know, that we may set $rs = 2m+1$ where $m>n+1$. Let us denote

\begin{equation}\label{piec}
 B= \frac{\widehat{\Psi}(2m+1)}{\widehat{\Psi}(2m)}
\end{equation}

Our aim is to prove, that $B>1$. We have

\begin{equation}\label{szesc}
B= \frac{\widehat{\Psi}(2m+1)}{\widehat{\Psi}(2m)} = \frac{\widehat{\Psi}(r\cdot s)}{\widehat{\Psi}(2m)}= \frac{\widehat{x}_r\cdot
\widehat{x}_{s}}{\widehat{x}_2\cdot\widehat{x}_m}= \frac{\widehat{x}_{r}}{\widehat{x}_{r-1}} \cdot \frac{\widehat{x}_{r-1}\cdot \widehat{x}_{s}}{\widehat{x}_2
\cdot \widehat{x}_m }=
\end{equation}

\begin{equation}\label{siedem}
=\frac{\widehat{x}_{r}}{\widehat{x}_{r-1}} \cdot \frac{\widehat{x}_{2}\cdot \widehat{x}_{t}\cdot \widehat{x}_{s}}{\widehat{x}_2 \cdot \widehat{x}_m}=
\frac{\widehat{x}_{r}}{\widehat{x}_{r-1}}\cdot \frac{\widehat{x}_{ts}}{\widehat{x}_m}> \frac{\widehat{x}_{r}}{\widehat{x}_{r-1}}\cdot
\frac{\widehat{x}_{ts-t}}{\widehat{x}_{m-t}}.
\end{equation}

By a similar argument as before, we check that $ts-t = w(r-1)$ and $m-t = wr$. Since $\widehat{\Psi}$ is an algebraic isomorphism in the whole ${\mathbb
N}^{*}$, we have $\widehat{x}_{ts-t}=\widehat{x}_w\cdot \widehat{x}_{r-1}$ and $\widehat{x}_{m-t}= \widehat{x}_w\cdot \widehat{x}_r$

Let us observe some inequalities. Since we have proved that $st<n+1$ then $ts-t\leq n$ and hence $w\cdot(r-1)\leq n$ and by induction hypothesis, we have
$$\widehat{x}_{ts-t}=\widehat{x}_{w(r-1)}=\widehat{x}_w \cdot
\widehat{x}_{r-1}.$$ We must also bound $m-t$ from above. We have $rs>2m$. Thus $(r-1+1)s>2m$ and $(r-1)s+s>2m$. But $(r-1)=2t$ then $2ts +s >2m$. We have
proved that $s<n$ and $ts<n+1$. In consequence $2m <2ts +s < 2n+2 +n \leq 3n+1$. Hence $m-t<2n$ and we can use the induction hypothesis for $m-t=wr$. Hence

\begin{equation}\label{dziewiec}
B= \frac{\widehat{\Psi}(2m+1)}{\widehat{\Psi}(2m)}> \frac{\widehat{x}_{r}}{\widehat{x}_{r-1}}\cdot \frac{\widehat{x}_{ts-t}}{\widehat{x}_{m-t}} =
\frac{\widehat{x}_{r}}{\widehat{x}_{r-1}}\cdot \frac{\widehat{x}_{w(r-1)}}{\widehat{x}_{wr}} = \frac{\widehat{x}_{r}}{\widehat{x}_{r-1}}\cdot
\frac{\widehat{x}_{w}\cdot \widehat{x}_{r-1}}{\widehat{x}_{w}\cdot\widehat{x}_r}=1
\end{equation}

Finally we obtain
\begin{equation}\label{osiem}
\widehat{x}_2 \cdot \widehat{x}_{n+1}>\widehat{x}_r \cdot \widehat{x}_{t}=\widehat{\Psi}(2m+1) >\widehat{\Psi}(2m) = \widehat{x}_2 \cdot \widehat{x}_m .
\end{equation}

In consequence $\widehat{x}_{n+1} > \widehat{x}_m$. This means that $\widehat{x}_m \in [1,\widehat{x}_2,...,\widehat{x}_{2n}]$ so we may use the induction
hypothesis and we obtain $n+1>m$. But we know, that $m>n+1$ and this contradiction ends the proof of the inductive step, and at the same time, the proof of the
main theorem.

\maketitle
\section{ Some additional remarks}\label{sekcja4}

\begin{remark}\label{koniec1}

 As we have observed, the space of likens is big,
but there are only a few examples of likens which could be described as "suitable for counting". A natural method of obtaining such kind of examples is to
choose a subset  $ K\subset {\mathbb N}^{*}$ and consider the sub-semigroup ${\mathbb L}(K)$ generated by $K$ ( i.e. the smallest semigroup containing the set
$K$)  which is ordered by the order inherited from ${\mathbb N}^{*}$. Hence we obtain a liken, a sub-liken of ${\mathbb N}^{*}$. In particular we may consider
only the likens generated by the subsets of the set of prime numbers. Even this family, "small" compared to the family of all likens, is neverthless "rich",
since it contains a continuum of non-isomorphic likens. Theorem \ref{main} shows, that only one of these likens is convex and has the razor property.
\end{remark}
\vspace{3mm}

\begin{remark}\label{koniec2}

 It is commonly known, that the Cauchy functional
equation of the type $f(x\cdot y) = f(x) + f(y)$ has many "bad" solutions and only one (up to a  constant factor) "good" solution if we claim  $f$ to be
continuous (or monotone, or locally bounded etc.) and this solution is the logarithmic function.
 It follows from  Theorem \ref{main} that the condition of convexity  for likens
 together with the razor property  may be considered as a kind
of condition guaranteing the uniqueness of the logarithmic function.

\end{remark}

\end{document}